\documentclass[a4paper,11pt]{article}
\pdfoutput=1

\usepackage[utf8]{inputenc}
\usepackage[T1]{fontenc}

\usepackage{amsmath,amssymb,amsfonts,amsthm}
\usepackage{mathrsfs}
\usepackage[all]{xy}
\usepackage{hyperref}
\usepackage{geometry}

\title{Minimal surfaces and symplectic structures of moduli spaces}
\author{Brice Loustau}
\date{}

\newcommand{\R}{\mathbb{R}}
\newcommand{\C}{\mathbb{C}}

\newcommand{\HH}{\mathbb{H}}

\newcommand{\CPS}{\mathcal{CP}(S)}
\newcommand{\TS}{\mathcal{T}(S)}
\newcommand{\FS}{\mathcal{F}(S)}
\newcommand{\QFS}{\mathcal{QF}(S)}
\newcommand{\AFS}{\mathcal{AF}(S)}
\newcommand{\RS}{\mathcal{R}(S)}

\newcommand{\XS}{\mathcal{X}(S)}

\newcommand{\slC}{\mathfrak{sl}_2(\C)}

\newcommand{\CP}{\C\mathbf{P}^1}

\DeclareMathOperator{\hol}{\mathit{hol}}
\DeclareMathOperator{\I}{\mathit{I}}
\DeclareMathOperator{\II}{\mathit{I\!I}}
\DeclareMathOperator{\III}{\mathit{I\!I\!I}}
\DeclareMathOperator{\PSL}{\mathit{PSL}}
\DeclareMathOperator{\SO}{\mathit{SO}}
\DeclareMathOperator{\SL}{\mathit{SL}}

\theoremstyle{plain}
\newtheorem{theorem}{Theorem}[section]
\newtheorem{coro}[theorem]{Corollary}
\newtheorem{propo}[theorem]{Proposition}
\newtheorem{lemma}[theorem]{Lemma}

\theoremstyle{definition}

\begin{document}

\maketitle

Given a closed surface $S$ of genus at least 2, we compare the symplectic structure of Taubes' moduli space of minimal hyperbolic germs 
with the Goldman symplectic structure
on the character variety ${\cal X}(S, PSL(2,\C))$ and the affine cotangent symplectic structure on the space of complex projective structures
$\CPS$ given by the Schwarzian parametrization.
This is done in restriction to the moduli space of almost-Fuchsian structures by involving a notion of renormalized volume, 
used to relate the geometry of a minimal surface in a hyperbolic $3$-manifold to the geometry of its ideal conformal boundary.

\setcounter{tocdepth}{2}
\pdfbookmark[0]{Contents}{tablematieres}
\tableofcontents

\section{Introduction}

Let $S$ be a closed oriented surface of genus $g \geqslant 2$. Several moduli spaces associated to $S$ are known to enjoy
natural real or complex symplectic structures. Let us  mention here those of interest for the purpose of this paper, more precise definitions
will follow in the subsequent sections.

\begin{itemize}
 \item The Teichmüller space of $S$ denoted by ${\cal T}(S)$ is the deformation space of complex structures on $S$. It is equipped
 with a symplectic structure $\omega_{WP}$ 
 which is the Kähler form of the so-called Weil-Petersson metric. 
 In various ways, $\TS$ will be seen as a half-dimension ``slice'' of all the following deformation spaces.
 \item The character variety ${\cal X}(S,G)$ where $G$ is a reductive Lie group is a quotient of the 
 space of group homomorphisms $\operatorname{Hom}(\pi_1(S),G)$. Following Atiyah-Bott \cite{atiyahbott}, 
 Goldman \cite{goldmannature} showed that it enjoys a natural symplectic structure $\omega_G$
 that generalizes the Weil-Petersson symplectic structure on Teichmüller space 
 ($\TS$ can be identified as a subspace of the character variety
 when $PSL(2,\R) \hookrightarrow G$ by the uniformization theorem). When $G$ is a complex Lie group,
 $\omega_G$ is a complex symplectic form. We will be mostly interested in the case $G = \PSL(2,\C)$.
 \item Taubes' moduli space ${\cal H}$ of minimal hyperbolic germs (introduced by Taubes in \cite{taubes}) consists of pairs $(\I, \II)$
 of a first and second fundamental form on $S \hookrightarrow S \times \R$ for the germ of a hyperbolic metric on $S \times \R$ such that
 $S$ is minimally embedded. ${\cal H}$ comes naturally equipped with a real symplectic structure $\omega_{\cal H}$.
 \item The deformation space $\CPS$ of complex projective structures (or more generally $(X,G)$-structures) 
 has a natural complex symplectic structure $\omega_G$
 which can be identified with the Goldman symplectic structure of the
 character variety via the holonomy map $\hol : \CPS \to {\cal X}(S,PSL(2, \C))$.
\end{itemize}

The deformation space of almost-Fuchsian structures $\AFS$ (introduced
by Uhlenbeck in \cite{uhlenbeck})) will be central in this article, it can be seen as an 
open subspace of either the character variety ${\cal X}(S,PSL_2(\C))$, 
Taubes' moduli space ${\cal H}$ or the space of complex projective structures $\CPS$.
Using results of \cite{article1} and an \textit{ad hoc} notion of renormalized for almost-Fuchsian manifolds, we compare
the Goldman symplectic structure $\omega_G$ of the character variety restricted to $\AFS$ with the symplectic structure 
$\omega_{{\cal H}}$ on ${\cal H}$ and, indirectly, the canonical symplectic structure $\omega_{\mathrm{can}}$ on the cotangent
bundle to Teichmüller space $T^*\TS$.  We show in particular\footnote{
Note that $\omega_G$ is a complex symplectic structure whereas $\omega_{{\cal H}}$ is a real
symplectic structure.}:
\newtheorem*{main4}{Theorem \ref{main4}}
\begin{main4}
In restriction to $\AFS$,
\begin{equation}
\omega_{\cal H} = -\mathrm{Im}(\omega_G)~.
\end{equation}
\end{main4}
We explain the strategy in the remaining of this introduction.

\medskip
First we compare the symplectic structure $\omega_G$ on $\CPS$ to the canonical symplectic
structure $\omega_{\mathrm{can}}$ on $T^* \TS$ under the Schwarzian parametrization of $\CPS$. 
This is thoroughly addressed in \cite{article1}. Recall that there is a natural ``forgetful''
projection $p : \CPS \to \TS$ which assigns to a complex projective structure on $S$ its underlying complex structure. 
The Schwarzian parametrization turns $p : \CPS \to \TS$ into an affine holomorphic bundle
modeled on $T^* \TS$. As a consequence there is an identification $\tau^\sigma : \CPS \stackrel{\sim}{\to} T^*\TS$, 
but it is not canonical: it depends
on the choice of a ``zero section'' $\sigma : \TS \to \CPS$. 
The result proven in \cite{article1} that we will use is:
\newtheorem*{blob}{Theorem \ref{blob}}
\begin{blob}[\cite{article1}, Corollary 6.13]
Let $\sigma_{\cal F} : \TS \to \CPS$ be the Fuchsian section\footnote{
$\sigma_F$ assigns to a complex structure on $S$ the uniformized Fuchsian projective structure, see section \ref{TSCPS}}. Then
\begin{equation} (\tau^{\sigma_{\cal F}})^* \omega_{\mathrm{can}} = -i(\omega_G - p^* \omega_{WP})~.\end{equation}
\end{blob}

Next we turn our attention to Taubes' moduli space $\cal H$ and minimal surfaces in hyperbolic $3$-manifolds. 
An observation going back to Hopf
\cite{hopf} shows that if $\Sigma$ is a minimal surface in a hyperbolic $3$-manifold $M$, then the second fundamental form $\II_\Sigma$
is the real
part of a unique holomorphic quadratic differential $\varphi$ on $\Sigma$ with respect to the complex structure $[\I_\Sigma]$ 
defined by the conformal class
of the first fundamental form $\I_\Sigma$. This defines a map $\Psi : (\I, \II) \mapsto ([\I], \varphi)$ from ${\cal H}$ to $T^*\TS$ 
(recall that for
$X\in \TS$, $T_X^* \TS$ is naturally identified with the space of holomorphic quadratic differentials on the Riemann surface $X$). 
The natural symplectic structure $\omega_{{\cal H}}$ on ${\cal H}$ 
can be described as $\omega_T = \mathrm{Re}(\Psi^* \omega_{\mathrm{can}})$.

Uhlenbeck (\cite{uhlenbeck}) showed that if $M \approx S \times \R$ is equipped
with a quasi-Fuchsian hyperbolic structure that happens to be \textit{almost-Fuchsian}
(see section \ref{almostFuchsian}), then there is a unique minimal embedding $S \hookrightarrow M$. 
Taking the first and second fundamental forms of the 
minimal surface defines an embedding of the almost-Fuchsian space $\AFS$ in the Taubes moduli space ${\cal H}$.
The renormalized volume of $M$ is a ``finite renormalization'' of the otherwise infinite hyperbolic volume of $[\Sigma, \partial_\infty^+ M)$,
the end of $M$ comprised between the minimal surface $\Sigma$ and the ideal boundary $\partial_\infty^+ M$, 
and it is related to equidistant foliations 
in this end. These ideas are made precise in sections \ref{minsurf} and \ref{renornor}. 
Borrowing arguments mainly from \cite{KS06} we prove the following theorem:
\newtheorem*{brodu}{Theorem \ref{brodu}}
\begin{brodu}
Let $W : \AFS \to \R$ be the renormalized volume function on the almost-Fuchsian space
and $\Psi : \AFS \to T^*\TS$ be the restriction to $\AFS$ of the map $\Psi : {\cal H} \to T^*\TS$ above.
Then the differential of $W$, as a $1$-form on $\AFS$, is expressed as
\begin{equation}
dW = -\frac{1}{4} \mathrm{Re}\left[\Psi^* \xi + {(\tau^{\sigma_{\cal F}})}^* \xi\right]
\end{equation}
where $\xi$ is the canonical one-form on the complex cotangent space $T^*\TS$ (so that $\omega_{\mathrm{can}} = d\xi$).
\end{brodu}
Theorem \ref{main4} then directly follows: just take the exterior derivative of this identity and use 
Theorem \ref{blob} above.

\begin{center}
\rule{120pt}{.5pt}
\end{center}
\vspace{1em}

\emph{Note:} The study and comparison of the affine cotangent symplectic structures on $\CPS$,
Goldman's symplectic structure on $\XS$ and the ``complex Hamiltonian picture'' of $\QFS$ is addressed in \cite{article1}. 
Both articles are based on the author's PhD thesis (\cite{bricethesis}).

\bigskip

\emph{Structure of the paper:} Section \ref{TSCPS} discusses the affine cotangent symplectic structures on $\CPS$
and their relation with Goldman's symplectic structure on the character variety $\XS$. Section \ref{minsurf}
reviews minimal surfaces in hyperbolic $3$-manifolds, Taubes' moduli space with its symplectic structure
and almost-Fuchsian structures. Section \ref{renornor} introduces renormalized volume 
for almost-Fuchsian manifolds and contains the proof of our main result.

\bigskip

\emph{Acknowledgments:} This paper is based on part of the author's PhD thesis, which was supervised by Jean-Marc Schlenker. I wish
to express my gratitude to Jean-Marc for his kind advice. I would also like to thank Steven Kerckhoff, Francis Bonahon, Bill Goldman, 
David Dumas, Jonah Gaster, Andy Sanders, Cyril Lecuire, Julien Marché, among others with whom I have had helpful discussions.

\bigskip

The research leading to these results has received funding from the European Research Council under the {\em European Community}'s 
seventh Framework Programme (FP7/2007-2013)/ERC {\em  grant agreement}.
\section{Complex projective structures and symplectic structures}\label{TSCPS}

Throughout this paper, $S$ will be a connected, oriented, smooth, closed surface with genus $g \geqslant 2$.
We refer to \cite{article1} for a more detailed exposition of this section.

\subsection{\texorpdfstring{$\CPS$ and $\TS$}{{CP(S)} and {T(S)}}}\label{defts}

A \emph{complex projective structure} on $S$ is given as a maximal atlas of charts 
mapping open sets of $S$ into the complex 
projective line $\CP$ such that the transition maps are (restrictions of) projective linear transformations
(\textit{i.e.} Möbius transformations of the Riemann sphere)\footnote{In 
terms of geometric structures (see e.g. \cite{thurston}),
a complex projective structure is a $(\CP,\PSL_2(\C))$-structure.}.
The atlas is required to be compatible with the orientation and smooth structure on $S$.
The group $\mathrm{Diff}^+(S)$ of orientation-preserving diffeomorphisms of $S$ naturally 
acts on the set of all complex projective structures on $S$, the quotient $\CPS$ by the subgroup
$\mathrm{Diff}^+_0(S)$\footnote{
$\mathrm{Diff}^+_0(S)$ is the identity component in $\mathrm{Diff}^+(S)$,
consisting of diffeomorphisms that are homotopic to the identity.} 
is called the \emph{deformation space of complex projective structures} on $S$.

Since a complex projective atlas is in particular a holomorphic atlas, 
a complex projective structure defines an underlying complex structure on $S$. This observation 
yields a ``forgetful projection $p : \CPS \to \TS$, where the Teichmüller space 
$\TS = \{\mathrm{complex}~\mathrm{structures}~\mathrm{on}~S\} / \mathrm{Diff}^+_0(S)$ 
is the deformation space of complex structures on $S$.

Both $\CPS$ and $\TS$ are naturally complex manifolds,
of dimensions $6g-6$ and $3g-3$ respectively. Let us recall in particular
that the complex cotangent space $T_X^* \TS$ is naturally identified with the vector space
$Q(X)$ of holomorphic quadratic differentials on the Riemann surface $X$ (for any $X \in \TS$).
The ``forgetful projection'' $p: \CPS \to \TS$ is easily seen to be a holomorphic map. It is surjective
(a section is given by uniformization) 
and its fibers are equipped with an affine structure, as explained in the subsequent
paragraphs.

The \emph{Weil-Petersson metric} on $\TS$ globalizes the Weil-Petersson product of holomorphic quadratic differentials, given by
$\left<\varphi,\psi\right>_{WP} = -\frac{1}{4} \int_X \varphi \cdot {\sigma}^{-1} \cdot \overline{\psi}$ for any 
$\varphi, \psi \in Q(X)$, where ${\sigma}^{-1}$ is the dual current of the area form $\sigma$ for the Poincar\'e metric\footnote{Recall
that the Poincaré metric is the unique conformal hyperbolic metric on $X$, given by uniformization.
In a complex chart with values in the upper half-plane $z = x + i y : U \subset X \to \HH^{2}$, 
the tensor product $-\frac{1}{4} \varphi \cdot {\sigma}^{-1} \cdot \overline{\psi}$ reduces to the classical expression
$\begin{displaystyle}
y^2 \varphi(z) \overline{\psi(z)} dx \wedge dy
\end{displaystyle}$.}.
This defines a dual Hermitian metric $\left<\cdot,\cdot\right>_{WP}$ on $\TS$, first shown to be K\"ahler by Ahlfors \cite{ahlfors2}
and Weil. The K\"ahler form of the Weil-Petersson metric on $\TS$ is the real symplectic form 
$\omega_{WP} = - \mathrm{Im} \left<\cdot,\cdot\right>_{WP}$. \label{WP}

\subsection{Fuchsian and quasi-Fuchsian projective structures}\label{fqf}

Whenever a Kleinian group\footnote{{\it i.e.} a discrete subgroup of $\PSL_2(\C)$.} $\Gamma$ acts freely and properly 
on some open subset $U$
of the complex projective line $\CP$, the quotient surface $U/\Gamma$ inherits a complex projective structure.
This gives a variety (but not all) of complex projective surfaces, called embedded projective structures.

Fuchsian projective structures are a fundamental example of embedded projective structures. 
Given a marked complex structure $X$, the uniformization theorem provides a representation $\rho : \pi_1(S) \to \PSL_2(\R)$
such that $X \approx \HH^2/\rho(\pi_1(S))$ as Riemann surfaces (where $\HH^2$ is the upper half-plane).
$\HH^2$ can be seen as an open set (a disk) in $\CP$ and the Fuchsian group $\rho(\pi_1(S)) \subset \PSL_2(\R)$ is in particular a 
Kleinian group, so the quotient
$X \approx \HH^2/\rho(\pi_1(S))$ inherits a complex projective structure $Z$. This defines a section
\begin{equation}\label{fufuch}
 \sigma_{\cal F} : \TS \to \CPS
\end{equation}
to $p$, called the \emph{Fuchsian section}.
We call $\FS:=\sigma_{\cal F}(\TS)$ the (deformation) space of (standard) Fuchsian (projective) structures on $S$, 
it is an embedded copy of $\TS$ in $\CPS$.

\emph{Quasi-Fuchsian structures} are another important class of embedded projective structures.
Given two marked complex structures $(X^+,X^-) \in \TS \times {\cal T}(\overline{S})$\footnote{where
$\overline{S}$ is the surface $S$ with reversed orientation},
Bers' simultaneous uniformization theorem states that there exists a
unique representation $\rho : \pi_1(S) \stackrel{\sim}{\to} \Gamma \subset \PSL_2(\C)$ up to conjugation such that:
\begin{itemize}
 \item[$\bullet$] The limit set\footnote{The \emph{limit set} $\Lambda = \Lambda(\Gamma)$ 
 is defined as the complement in $\CP$ of the domain
of discontinuity $\Omega$, which is the maximal open set on which $\Gamma$ acts freely and properly. Alternatively,
$\Lambda$ is described as the closure in $\CP$ of the set of fixed points of elements of $\Gamma$.} $\Lambda$ is a Jordan curve.
 The domain of discontinuity $\Omega$ is then 
 the disjoint union of two simply connected domains $\Omega^+$ and $\Omega^-$. A such $\Gamma$ is called a
 quasi-Fuchsian group.
 \item[$\bullet$] As marked Riemann surfaces, $X^+ \approx \Omega^+/\Gamma$ and $X^- \approx \Omega^-/\Gamma$.
\end{itemize}
Again, both Riemann surfaces $X^+$ and $X^-$ inherit embedded complex projective structures $Z^+$ and $Z^-$ by this construction.
This defines a map
$\beta  = (\beta^+,\beta^-) : 
\TS \times  {\cal T}(\overline{S}) \to \CPS \times  {\cal CP}(\overline{S})$
which is a holomorphic section to $p \times p : \CPS \times  {\cal CP}(\overline{S}) \to  \TS \times  {\cal T}(\overline{S})$
by Bers' theorem.

In particular, when $X^- \in {\cal T}(\overline{S})$ is fixed, the map 
$\sigma_{X^-} := \beta^+(\cdot,X^-) : \TS \to \CPS \label{bbss}$
is a holomorphic section to $p$, called a \emph{Bers section}. \label{berssection}
$\QFS:=\beta^+(\TS \times  {\cal T}(\overline{S})) \subset \CPS$ \label{qfdef} is called the (deformation) space of
(standard) quasi-Fuchsian (projective) structures on $S$. It is an open neighborhood of $\FS$ in $\CPS$.

\subsection{The character variety}\label{charvar}

References for this section include \cite{goldmannature}, \cite{porti}, \cite{Gsl2} and \cite{dumas}.

\subsubsection*{The character variety and Goldman's symplectic structure}

Let $G = \PSL_2(\C)$ and $\RS$ be the set of group homomorphisms from $\pi := \pi_1(S)$ to $G$.
It has a natural structure of a complex affine algebraic set\footnote{notably because there
is an isomorphism $\PSL_2(\C) \approx \SO_3(\C)$
(given by the adjoint representation of
$\PSL_2(\C)$ on its Lie algebra $\mathfrak{g} = \slC$).}. $G$ acts algebraically on $\RS$ by conjugation, and 
the character variety ${\cal X}(S)$ is defined as the quotient in the sense of invariant theory
$\XS = \RS // G$\footnote{Specifically, 
the action of $G$ on $\RS$ induces an action on the ring of regular functions $\C[\RS]$. Denote by $\C[\RS]^G$ the ring of invariant
functions, it is finitely generated because $\RS$ is affine and $G$ is reductive. In fact, it is generated in this case 
($G=\PSL_2(\C)$) by a finite number of the complex valued functions on $\RS$ of the form 
$\rho \mapsto \mathrm{tr}^2(\rho(\gamma))$ (see \cite{porti}). $\XS$ is the affine set such that $\C[\XS] = \C[\RS]^G$.}.
The points of $\XS$ are in one-to-one correspondence with
the set of \emph{characters}, {\it i.e.} complex-valued functions, defined on $\pi$, of the form $\gamma \in \pi \mapsto \mathrm{tr}^2(\rho(\gamma))$.
The affine set $\XS$ splits into two irreducible components $\XS_l \cup \XS_r$, where elements of $\XS_l$
are characters of representations that lift to $\SL(2,\C)$.

The set-theoretic quotient $\RS /G$ is rather complicated, but $G$ acts freely and properly on the subset $\RS^s$ of irreducible\footnote{
A representation $\rho : \pi \to \PSL_2(\C)$ is called \emph{irreducible} if it fixes no point in $\CP$.} (``stable'') 
representations, so that the
quotient $\RS^s/G$ is a complex manifold.
Furthermore, an irreducible representation is determined by its character, so that $\XS^s := \RS^s/G$ embeds (as a Zariski-dense open subset) in the smooth locus
of $\XS$. Its dimension is $6g-6$.

By the general construction of Goldman, the character variety enjoys a complex symplectic structure.
We refer to \cite{goldmannature} for this construction, but here is how it is defined in a nutshell.
At least for a stable point $[\rho] \in \XS^s$, the tangent space
is identified in terms of group cohomology as the space $H^1(\pi,\mathfrak{g}_{\textrm{Ad} \circ \rho})$. Here
$\mathfrak{g}$ is the Lie algebra of $G$, seen as a $\pi$-module by $\textrm{Ad} \circ \rho : \pi \to \mathrm{Aut}(\mathfrak{g})$.
Recall that the Lie algebra $\mathfrak{g} = \slC$ is equipped with its complex Killing form $B$\footnote{It is a non-degenerate
complex bilinear symmetric form preserved
by $G$ under the adjoint action, explicitly given by $B(u,v) = 4\mathrm{tr}(uv)$ where $u, v\in \slC$ 
are represented by trace-free $2 \times 2$ matrices.}.
The symplectic pairing of tangent vectors is given by taking the cup-product in group cohomology with $B/4$ as ``coefficient pairing'':
\begin{equation}H^1(\pi,\mathfrak{g}_{\textrm{Ad} \circ \rho}) \times H^1(\pi,\mathfrak{g}_{\textrm{Ad} \circ \rho}) \stackrel{\cup}{\to}
H^2(\pi,\mathfrak{g}_{\textrm{Ad} \circ \rho} \otimes \mathfrak{g}_{\textrm{Ad} \circ \rho}) \stackrel{B/4}{\to} H^2(\pi,\C) \cong \C ~.\end{equation}

By arguments of Goldman (\cite{goldmannature}) following Atiyah-Bott (\cite{atiyahbott}),
the complex $2$-form obtained this way is closed, in other words it is a complex symplectic form (at least on the 
smooth quasi-affine variety $\XS^s$ of stable points). \label{gsdef}

\subsubsection*{Holonomy of projective structures}\label{holonomy}

Just like any geometric structure, a complex projective structure $Z$ defines a \emph{developing map}
and a \emph{holonomy representation} (see e.g. \cite{thurston}). The developing map is a locally injective projective map 
$f : \tilde{Z} \rightarrow \CP$ 
and it is equivariant with respect to the holonomy representation $\rho : \pi \rightarrow \PSL_2(\C)$ in the sense that 
$f \circ \gamma = \rho(\gamma) \circ f$ for any 
$\gamma \in \pi$.

Holonomy of complex projective structures defines a map
\begin{equation*}
\hol : \CPS \rightarrow \XS
\end{equation*} 
which turns out to be a local biholomorphism, but neither injective nor a covering onto its image (\cite{hejhal}).
Nonetheless, we get a complex symplectic structure on $\CPS$ simply by pulling back that of $\XS^s$ by the holonomy map. 
Abusing notations, we will still
call this symplectic structure $\omega_G$\footnote{Alternatively, one could directly define $\omega_G$ on $\CPS$ 
in terms of the exterior product of $1$-forms with 
values in some flat bundle. The tangent space to $\CPS$ at a point $Z$ can be identified as $T_Z \CPS = \check{H}^1(Z,\Xi_Z)$, 
where $\Xi_Z$ is the sheaf of projective vector fields on $Z$.}.
We will consider $\omega_G$ as the \emph{standard} complex symplectic structure on $\CPS$.

The holonomy map is however injective in restriction to the space of standard quasi-Fuchsian projective structures $\QFS$,
so that $\QFS$ and \textit{a fortiori} $\FS$ are embedded in $\XS$.
Note that the Fuchsian space $\FS$ is the deformation space of hyperbolic structures on $S$, under the holonomy
map it is identified as the component
of discrete and faithful representations in the real character variety ${\cal X}(S,\PSL_2(\R))$. In \cite{goldmannature},
Goldman also shows that the symplectic form $\omega_G$ restricts to the Weil-Petersson Kähler form $\omega_{WP}$ on this component.
More precisely, this is stated as follows in our setting:

\begin{theorem}[Goldman \cite{goldmannature}]
Recall that $\sigma_{\cal F} : \TS \to \CPS$ denotes the Fuchsian section. Then
\begin{equation}\label{wgfsss}
 {(\sigma_{\cal F})}^* \omega_{G} = \omega_{WP}~.
\end{equation}
\end{theorem}

\subsection{\texorpdfstring{Schwarzian parametrization of $\CPS$ and the affine cotangent symplectic structures}
{Schwarzian parametrization of {CP(S)} and the affine cotangent symplectic structures}} \label{affinebundle}

\subsubsection*{Schwarzian parametrization}

The Schwarzian parametrization turns the space $\CPS$ with its projection $p : \CPS \to \TS$ into
a holomorphic affine bundle modeled on the holomorphic cotangent vector bundle $\pi : T^*\TS \to \TS$. We very briefly sketch
how this works, and refer to e.g. \cite{dumas} for a detailed exposition (the details of this construction will not be required
for the purpose of this paper).

Given a locally injective holomorphic function $f: Z_1 \rightarrow Z_2$ where $Z_1$ and $Z_2$ are complex projective surfaces, 
define the \emph{osculating map} $\tilde{f}$ to $f$ at a point $m\in Z_1$
as the germ of a (locally defined) projective map that has the best possible contact with $f$ at $m$. In some sense, one can take
a flat covariant derivative $\nabla \tilde{f}$ and identify it as holomorphic quadratic differential ${\cal S}(f) \in Q(X)$, 
called the \emph{Schwarzian derivative} \label{defschwarzian} of $f$.

Let $X$ be a fixed point in $\TS$ and $P(X) := p^{-1}(\{X\})$ the set of marked projective structures on $S$ whose underlying
complex structure is $X$.
Given $Z_1$, $Z_2 \in P(X)$, the identity map $\mathrm{id}_S : Z_1 \rightarrow Z_2$ is holomorphic but not projective if $Z_1 \neq Z_2$. 
Taking
its Schwarzian derivative $\varphi = {\cal S}(\mathrm{id}_S) \in Q(X)$ accurately measures the ``difference $Z_2 - Z_1$'' 
between the two projective structures 
$Z_1$ and $Z_2$. It turns
out that the properties of the Schwarzian derivative ensure that $Z_1, Z_2 \in P(X) \mapsto \varphi \in Q(X)$ 
indeed equips $P(X)$ with the structure of a complex affine space modeled on the vector space $Q(X)$.

Recall that $Q(X)$ is also identified with the complex cotangent space $T_X^*\TS$. As a result of this discussion, 
$p : \CPS \to \TS$ is an affine holomorphic bundle modeled on the holomorphic cotangent vector bundle $T^* \TS$.
As a consequence, $\CPS$ can be identified with $T^* \TS$ by choosing a ``zero section'' $\sigma : \TS \to \CPS$.
Explicitly, we get an isomorphism of complex affine bundles $\tau^{\sigma} : Z \mapsto Z - \sigma\left(p(Z)\right)$ 
such that $\tau^\sigma \circ \sigma$ is the zero section to $\pi : T^* \TS \to \TS$.
It is an isomorphism of holomorphic bundles whenever $\sigma$ is a holomorphic section to $p$,
such as a Bers section.

\subsubsection*{\texorpdfstring{Cotangent affine symplectic structures}{Cotangent affine symplectic structures}}\label{compsympcot}

Recall that for any complex manifold $M$ (in particular $M=\TS$), the total space of its 
holomorphic cotangent bundle $T^*M$ is equipped
with a canonical holomorphic $1$-form $\xi$ and a canonical complex symplectic structure $\omega_{\mathrm{can}} = d\xi$.

As we have seen in the previous paragraph, any choice of a ``zero section'' $\sigma : \TS \to \CPS$ yields an affine isomorphism 
$\tau^{\sigma} : \CPS \stackrel{\sim}{\to} T^* \TS$. We can use this to pull back the canonical symplectic structure of $T^*\TS$ on $\CPS$:
define \begin{equation}\omega^\sigma := (\tau^{\sigma})^* \omega_{\mathrm{can}}~.\end{equation}.

How is $\omega^\sigma$ affected by the choice of the ``zero section'' $\sigma$? 
A small computation (see \cite{article1}) shows that:
\begin{propo}\label{craca}
 For any two sections $\sigma_1$ and $\sigma_2$ to $p: \CPS \to \TS$,
\begin{equation}\omega^{\sigma_2} - \omega^{\sigma_1} = - p^* d(\sigma_2 - \sigma_1)\end{equation}
\end{propo}
\noindent where $\sigma_2 - \sigma_1$ is the ``affine difference'' between $\sigma_2$ and $\sigma_1$, it is a $1$-form on $\TS$.

Now, McMullen proved in \cite{mcmullenkahler} the following theorem:
\begin{theorem}[McMullen \cite{mcmullenkahler}]\label{mcm}
If $\sigma$ is any Bers section, then \begin{equation}d(\sigma_{\cal F} - \sigma) = -i\omega_{WP}~.\end{equation}
\end{theorem}

A consequence of this is that if $\sigma$ is chosen among Bers sections, $\omega^\sigma$ does not depend on $\sigma$.
Moreover, $\omega^\sigma$ is a complex symplectic form on $\CPS$ which restricts to $-i \omega_{WP}$
on the Fuchsian slice $\FS$, where $\omega_{WP}$ is the Weil-Petersson Kähler form on $\TS$. On the other hand, 
the standard symplectic structure $\omega_G$ is a complex symplectic
structure on $\CPS$ which restricts to $\omega_{WP}$ on the Fuchsian slice (see \ref{wgfsss}).
An analytic continuation argument ensures that:

\begin{theorem}[\cite{article1}, Theorem 6.10]\label{main1}
If $\sigma : \TS \rightarrow \CPS$ is any Bers section, then \begin{equation}\omega^\sigma = -i \omega_G~.\end{equation}
\end{theorem}

We also get the expression of the $2$-form $\omega^{\sigma_{\cal F}}$:

\begin{coro}[\cite{article1}, Corollary 6.13]\label{blob}
Let $\sigma_{\cal F} : \TS \to \CPS$ be the Fuchsian section. Then
\begin{equation}\omega^{\sigma_{\cal F}} = -i(\omega_G - p^* \omega_{WP})~.\end{equation}
\end{coro}

We refer to \cite{article1} for details.
\section{Minimal surfaces in hyperbolic 3-manifolds, Taubes' moduli space and almost-Fuchsian structures} \label{minsurf}

Minimal surfaces in hyperbolic $3$-manifolds have generated a lot of attention, with important contributions by Uhlenbeck (\cite{uhlenbeck})
and more recently Taubes (\cite{taubes}). 
In this paper, a primary source of inspiration was the work of Krasnov-Schlenker (\cite{KSmin}, \cite{KS06}, \cite{KSmin}).

\subsection{Minimal surfaces in hyperbolic 3-manifolds}

\subsubsection*{Extrinsic invariants of surfaces in hyperbolic 3-manifolds and the Gauss-Codazzi equations}

Consider an oriented immersed surface $\Sigma$ in an oriented hyperbolic $3$-manifold $M$. 
Denote by $g$ the hyperbolic metric on $M$ and $\bar{\nabla}$ its 
Levi-Civita connection. The following are classical extrinsic invariants of $\Sigma$:
\begin{itemize} \label{deffund}
 \item[$\bullet$] The \emph{first fundamental form} $\I$ is the Riemannian metric on $\Sigma$ induced by the hyperbolic metric $g$.
 Let $\nabla$ denote its Levi-Civita connection and $K$ its curvature.
 \item[$\bullet$] The \emph{shape operator} $B \in \mathrm{End}(T \Sigma)$ is defined by $B v := - \bar{\nabla}_v n$, 
 where $n$ is the positively oriented unit normal vector field to $\Sigma$.
 It is self-adjoint with respect to $\I$. The \emph{mean curvature} is defined by $H:= \mathrm{tr}(B)$.
 \item[$\bullet$] The \emph{second fundamental form} $\II$ is the symmetric bilinear form associated to $B$ with respect to 
 $\I$: $\II(u,v) := \I(Bu,v) = \I(u,Bv)$. We use the following
 notation convention: $B = \I^{-1}\II$.
 \item[$\bullet$] The \emph{third fundamental form} $\III$ is the symmetric bilinear form defined by $\III(u,v) = \I(Bu,Bv)$.
\end{itemize}

These satisfy the \emph{Gauss-Codazzi equations} on $\Sigma$:
\begin{equation}\left\{\begin{array}{ll}
   \det B = K+1 & \textrm{(Gauss equation)}\\
   d^\nabla B = 0 & \textrm{(Codazzi equation)}\\
  \end{array}\right.\end{equation}
where $B$ is seen as a $T\Sigma$-valued one-form in the Codazzi equation and $d^\nabla$ is the extension of the exterior derivative using the connection $\nabla$.

Conversely, the ``fundamental theorem of surface theory'' states that if $\I$ is a Riemannian metric on a surface 
$S$ and $\II$ is a symmetric bilinear
form such that $\I$ and $\II$ satisfy the Gauss-Codazzi equations, then there exists an isometric immersion of 
$(S, \I)$ in a possibly non-complete
hyperbolic $3$-manifold $M$ such that $\II$ is the second fundamental form of the immersed surface.

\subsubsection*{Minimal surfaces in hyperbolic $3$-manifolds and holomorphic qua\-dra\-tic dif\-fe\-ren\-tials} \label{minis}

A \emph{minimal surface} in a hyperbolic $3$-manifold $M$ is a minimally\footnote{In the
sense that it is an extremal of the area functional. Equivalently, the mean
curvature $H$ of the immersed surface vanishes identically. We refer to e.g.
\cite{chen} for general background on minimal submanifolds of Riemannian manifolds.} isometrically immersed Riemannian surface $(S,\I)$ in $M$. 
By the ``fundamental
theorem of surface theory'', this is equivalent to the three following conditions on the extrinsic invariants of the immersed surface:
\begin{equation}\left\{\begin{array}{ll}
   \det B = K+1 \\
   d^\nabla B = 0 \\
   H = 0~.
  \end{array}\right.\end{equation}

The following lemma was first discovered by Hopf and is quite straightforward to prove but it 
provides a surprising relation between minimal surfaces and holomorphic quadratic differentials:
\begin{lemma}[\cite{hopf}]
Let $S$ be an oriented surface equipped with a Riemannian metric $\I$ and a symmetric bilinear form $\II$ on $S$.
Let $B:=\I^{-1}\II$. Consider the
conformal class $[\I]$, so that $X := (S,[\I])$ is a Riemann surface.
\begin{enumerate}
 \item[$(i)$] $\II$ is the real part of a (unique) smooth quadratic differential $\varphi$ if and only if $\mathrm{tr}(B) = 0$.
 \item[$(ii)$] If $(i)$ holds, then $\varphi$ is holomorphic on $X$ if and only if $d^\nabla B = 0$.
\end{enumerate}
\end{lemma}

In particular, any embedded minimal surface $S \hookrightarrow M$ in a hyperbolic $3$-manifold defines a Riemann surface 
$X :=(S,[\I])$
and a holomorphic quadratic differential $\varphi\in Q(X)$, {\it i.e.} a point in the holomorphic cotangent of the Teichm\"uller 
space of $S$
(see section \ref{defts}).

\subsection{Taubes' moduli space} \label{tobi}

A \emph{minimal hyperbolic germ} is a pair $(\I, \II)$ where $\I$ is a (smooth) Riemannian metric on $S$ and $\II$ is a symmetric bilinear form
on on $S$ such that $\det B = K+1$, $d^\nabla B = 0$ and $H = 0$, where the notations correspond to the ones above: $K$ is the curvature
of $\I$, $B = \I^{-1}\II$ (meaning that $B$ is the $\I$-self-adjoint endomorphism of $TS$ associated to $\II$), $\nabla$ is the
Levi-Civita connection of $\I$, $d^\nabla$ is the extension of $\nabla$ to $TS$-valued $1$-forms and $H = \mathrm{tr}(B)$.
Alternatively, one could define a minimal hyperbolic germ as a pair $(\I, \varphi)$, where $\I$ is a Riemannian
 metric on $S$ and $\varphi$ is a holomorphic quadratic differential (with respect to the complex structure on $S$ determined
 by the conformal class of $\I$ and the orientation of $S$) such that $-\Vert \varphi \Vert_{\I}^2 = K + 1$. The relation
 with the previous definition is just $\II = \mathrm{Re}(\varphi)$, and one can check that $-\Vert \varphi \Vert_{\I}^2 = \det B$.
More geometrically, a minimal hyperbolic germ is a pair $(\I, \II)$ such that $\II$ is the second fundamental form of a minimal
 isometric embedding of $(S, \I)$ into a (possibly non-complete) hyperbolic $3$-manifold.

\emph{Taubes' moduli space} is then defined as the quotient of the set of all hyperbolic germs by the natural action of
the group $\mathrm{Diff}^+_0(S)$. In $\cite{taubes}$, Taubes shows that this quotient only has smooth points (in a reasonable sense),
so that ${\cal H}$ is a smooth manifold (of dimension $12g - 12$).

The map which assigns to a minimum hyperbolic germ $(I, \II)$ the couple $(X, \varphi)$ where $X$ is the complex
structure on $S$ determined by the conformal class of $\I$ and the orientation of $S$ and $\varphi \in Q(X)$ is the holomorphic
quadratic differential on $X$ such that $\II = \mathrm{Re}(\varphi)$ 
descends to the quotient as a well-defined smooth map $\Psi : {\cal H} \to T^*\TS$ (recall
that $T_X^* \TS \approx Q(X)$).

Taubes also shows that ${\cal H}$ naturally enjoys a real symplectic structure $\omega_{\cal H}$, obtained by a symplectic
reduction of the space $T^* \mathrm{Met}(S)$ equipped with its canonical symplectic structure (here $\mathrm{Met}(S)$ denotes
the space of smooth Riemannian metrics on $S$). Conveniently, this symplectic structure can alternatively be described as
$\omega_{\cal H} = \mathrm{Re}(\Psi^*\omega_{\mathrm{can}})$\footnote{
Taubes' definition of $\omega_{\cal H}$ probably differs from ours by a factor $2$, because the canonical real symplectic structure on the 
real cotangent bundle ${T_{\R}}^*\TS$ and the real part of the canonical complex symplectic structure on the holomorphic cotangent
bundle $T^*\TS$ differ by a factor $2$ under the usual identification ${T_{\R}}^*\TS \approx T^*\TS$.}.

\subsection{Almost-Fuchsian structures} \label{almostFuchsian}

A minimal hyperbolic germ $(\I, \II)$ is called \emph{almost-Fuchsian} when $\Vert \II \Vert_{\I}^2 < 2$ everywhere\footnote{or 
equivalently $\det B > -1$, or $\Vert \varphi  \Vert_{\I}^2 < 1$. See previous paragraph for the notations $B$ and $\varphi$.}.
Said differently, $\II$ is the second fundamental form for a minimal isometric embedding of $\Sigma$ in a hyperbolic $3$-manifold
such that its principal curvatures are everywhere in $(-1,1)$. The \emph{almost-Fuchsian space} $\AFS$ is defined as the subspace 
of ${\cal H}$ of classes of almost-Fuchsian minimal hyperbolic germs.

This definition is motivated by the observation of Uhlenbeck \cite{uhlenbeck} that 
the germ of the hyperbolic metric associated to an almost-Fuchsian germ can be extended to a complete
hyperbolic metric on $S \times \R$. More precisely, if $(\I, \II)$ is an almost-Fuchsian
minimal hyperbolic germ, then there is a unique quasi-Fuchsian $3$-manifold $M$\footnote{
A quasi-Fuchsian $3$-manifold is a manifold isometric to $\HH^3 / \Gamma$ where
$\Gamma$ is a quasi-Fuchsian group as in \ref{fqf}.} up to isometry with a minimal
isometric embedding $(S,\I) \hookrightarrow M$ giving $\II$ as the second fundamental form.

As a consequence, almost-Fuchsian space $\AFS \subset {\cal H}$ embeds as an open set of $\QFS$. Recall that the space
of quasi-Fuchsian structures $\QFS$ can be seen either as a subspace of the space of complex projective structures $\CPS$ or the
character variety $\XS$. In turn, we will (somewhat abusively) think of almost-Fuchsian structures as living either in 
${\cal H}$, $\CPS$ or $\XS$, depending on the context.

Note that the Fuchsian space $\FS$ corresponds to the ``Fuchsian minimal hyperbolic germs'', {\textit i.e.} of the form
$(\I, 0)$. The map $\Psi : {\cal H} \mapsto T^*\TS$ restricts to ``the identity'' on the space of Fuchsian germs. In other words,
it identifies the space of Fuchsian germs ($ \approx \FS$) with the zero section of $T^*\TS$ ($\approx \TS \approx
\FS$ by uniformization).
\section{Renormalized volume of almost-Fuchsian manifolds} \label{renornor}

In this section we introduce an \textit{ad hoc} notion of renormalized volume for almost-Fuchsian manifolds. Studying its infinitesimal
variations will enable us to compare the symplectic structure $\omega_{\cal H}$ on the almost-Fuchsian space to the symplectic structures
encountered in section \ref{TSCPS}.
We refer to \cite{KS06} and \cite{KSsurvey} for a systematic presentation of the renormalized volume of hyperbolic $3$-manifolds.

\subsection{Equidistant foliations in quasi-Fuchsian 3-manifolds}

Let $M$ be a quasi-Fuchsian $3$-manifold.  Consider a smooth embedded surface $S_0$ such that the normal exponential
map induces an isotopic deformation of $S_0$ on the ideal boundary component $\partial_\infty^+ M$. In particular, 
$S_0$ disconnects
$M$ (or equivalently $\hat{M} = M \cup \partial_\infty M$) into two connected components, each one containing a boundary component.
Denote by $\left[S_0,\partial_\infty^+ M\right]$ 
(resp. $\left[\partial_\infty^-M, S_0\right]$) the connected component of $\hat{M}$ containing $\partial_\infty^+ M$ 
(resp. $\partial_\infty^- M$).
Let $S_\rho$ be the set
of points at distance $\rho \geqslant 0$ of $S_0$ in $\left[S_0,\partial_\infty^+ M\right]$. 
We assume now that $S_0$ is convex, in the sense that 
$\left]\partial_\infty^-M, S_0\right]$ is
geodesically convex in $M$. The nearest-point projection $\kappa : \left[S_0,\partial_\infty^+ M\right[ \to S_0$ is
well-defined and smooth, and it admits a natural
extension to $\partial_\infty^+ M$. It is easy to see that $S_\rho$ is a smooth embedded surface, obtained as the image of a section
$u_\rho : S_0 \rightarrow S_\rho$ to $\kappa$\footnote{Note that
$\kappa$ is a right inverse to the normal exponential map, and $u_\rho$ is just the time-$\rho$ normal exponential map.}
(including in the case $\rho = +\infty$, with $S_\infty = \partial_\infty^+ M$). The $1$-parameter
family $(S_\rho)_{\rho \geqslant 0}$ is called an equidistant foliation of the topological end $\partial^+ \hat{M}$, 
it is a smooth foliation of 
$\left[S_0,\partial_\infty^+ M\right]$ by equidistant surfaces. Two equidistant foliations $(S_\rho)_{\rho \geqslant 0}$ and
$(S'_\rho)_{\rho \geqslant 0}$
are considered equivalent if $S_0 = S'_{\rho}$ for some $\rho>0$ or vice-versa.

Let $(S_\rho)_{\rho \geqslant 0}$ be an equidistant foliation as above. For any $\rho\geqslant 0$, the map 
$f_\rho = {u_\rho} \circ {u_\infty}^{-1}$ is a 
smooth identification $f_\rho : S_\infty \stackrel{\sim}{\to} S_\rho$. 
Let $\I_\rho := {f_\rho}^* \I_{S_\rho}$ where $\I_{S_\rho}$ is the first
fundamental form on $S_\rho$, and define $\II_\rho$ and $\III_\rho$ similarly. 
Writing down the differential equation satisfied by $\I_\rho$ and integrating it,
one shows (see \cite{KSmin}):
\begin{propo}
For any $\rho,h>0$, \begin{equation}\I_{\rho+h} = (\cosh^2 h) \I_\rho + 2 (\sinh h \cosh h) \II_\rho + (\sinh^2 h) \III_\rho~.\end{equation}
\end{propo}

The following easily follows:
\begin{propo}~
\begin{itemize}
 \item[$\bullet$] $\I_\rho \sim_{\rho \rightarrow \infty} \frac{e^{2\rho}}{2} \I^*$, where $\I^*:= \frac{1}{2}(\I_0+2\II_0+\III_0)$.
 \item[$\bullet$] For any $\rho\geqslant0$, $\I^* = \frac{e^{-2 \rho}}{2}(\I_\rho + 2 \II_\rho + \III_\rho)$.
 \item[$\bullet$] $\II^* := \frac{1}{2}(I_\rho - \III_\rho)$ does not depend on $\rho$.
\end{itemize}
\end{propo}

A standard argument shows that $I^*$ must belong to the conformal class of the ideal boundary $\partial_\infty^+ M$. 
Furthermore, a theorem first proved
by C. Epstein (see also \cite{HRc}) says that:
\begin{theorem}[C. Epstein]
Given any metric $g$ in the conformal class of $\partial_\infty^+ M$, there is a unique equidistant foliation up to 
equivalence $(S_\rho)_{\rho\geqslant0}$ such that $\I^* = g$.
\end{theorem}

\subsection{Definition of the renormalized volume}

Suppose now that $M$ is almost-Fuchsian, so that it contains a unique minimal surface $\Sigma$. 
Let $(S_\rho)_{\rho\geqslant0}$ be the
unique equidistant foliation associated to the Poincar\'e metric on $\partial_\infty^+ M$\footnote{\textit{i.e.}
$I_* = \frac{e^{-2 \rho}}{2}(\I_\rho + 2 \II_\rho + \III_\rho)$ (for any $\rho$) is the unique
conformal hyperbolic metric on $\partial_\infty^+ M$.}. 
Denote by $N_\rho = \left[\Sigma,S_\rho\right]$ be the
compact connected component of $M \setminus (\Sigma \cup S_\rho)$ and by $V(N_\rho)$ its hyperbolic volume. 
A direct calculation shows that:
\begin{propo}
The number \begin{equation}W = V(N_\rho) - \frac{1}{4}\left(\int_{S_\infty}H_\rho da_\rho\right) - 2\pi(g-1) \rho\end{equation} 
does not depend on $\rho$. NB: here $H_\rho$ is the mean curvature $H_\rho = \mathrm{tr}(\I_\rho^{-1} \II_\rho)$ and 
$da_\rho$ is the area element for the metric $\I_\rho$.
\end{propo}
\noindent $W$ will be called the \emph{renormalized volume} of the almost-Fuchsian manifold $M$.
Renormalized volume thus defines a function $W : \AFS \to \R$.

\subsection{Variation of the renormalized volume, comparing symplectic structures}

Next we want to compute the variation of the renormalized volume under an infinitesimal variation of the hyperbolic metric
(in other words we want to compute the differential $dW$).

The following formula for the variation of volume was proved by Rivin-Schlenker \cite{RS99}:
\begin{theorem}[Rivin-Schlenker \cite{RS99}]
Let $M$ be a convex cocompact hyperbolic $3$-manifold and let $N\subset M$ be a convex compact subset of $M$ with smooth boundary.
Under an infinitesimal deformation of the hyperbolic metric on $M$,
\begin{equation}2 \delta V(N) = \int_{\partial N} \left(\delta H + \frac{1}{2}\left< \delta \I, \II \right>_{\I} \right)da_{\I}~. \label{qmq}\end{equation}
\end{theorem}

We use this to show:
\begin{theorem}
Let $M$ be an almost-Fuchsian manifold. Under an infinitesimal deformation of the hyperbolic metric on $M$, the variation of the
renormalized volume is given by 
\begin{equation}
\delta W = -\frac{1}{4} \int_{S_\infty} \left<\II_\Sigma, \delta \I_\Sigma\right>_{\I_\Sigma}da_\Sigma -
\frac{1}{4} \int_{S_\infty} \left<\II_0^*,\delta\I^*\right>_{\I^*} da^*
\end{equation}
where $\II_0^*$ is the trace-free part of $\II^*$ (with respect to ${\I^*}$).
\end{theorem}

\begin{proof}
We only give an outline of the proof; most of the calculations needed have already been done in \cite{KS06}. 
Following the definition of $W$, one has
$$\delta W = \delta V(N_\rho) - \frac{1}{4}\left(\int_{S_\infty}\delta H_\rho da_\rho\right) - \frac{1}{4}\left(\int_{S_\infty}H_\rho \delta (da_\rho)\right)$$
Using the formula \ref{qmq}, the variation of $V(N_\rho)$ is given by
$$\delta V(N_\rho) = -\frac{1}{4} \left(\int_{S_\infty} \left<\delta \I_\Sigma, \II_\Sigma \right>_{\I_\Sigma} \right)da_{\Sigma}
+ \frac{1}{2}\left(\int_{S_\infty} \delta H_\rho + \frac{1}{2}\left< \delta \I_\rho, \II_\rho \right>_{\I_\rho} da_{\rho}\right)~.$$
The variation of the mean curvature $H$ on a surface is $\delta H 
= -\left<\delta\I,\II\right>_{\I}+\left<\I,\delta\II\right>_{\I}$, and the variation of the area element is given by 
$\delta(da)= \frac{1}{2}\left<\delta \I,\I\right>_{\I} da$. Putting all this together, one finds
$$\delta W = -\frac{1}{4} \left(\int_{S_\infty} \left<\delta \I_\Sigma, \II_\Sigma \right>_{\I_\Sigma} \right)da_{\Sigma}
+ \frac{1}{4}\left(\int_{S_\infty} \left< \delta \II_\rho - \frac{H_\rho}{2}\delta \I_\rho, \I_\rho \right>_{\I_\rho} da_{\rho}\right)~.$$
A tedious but straightforward calculation done in \cite{KS06} (Corollary 6.2) shows that for any foliation $(S_\rho)_{\rho\geqslant 0}$,
$$ \int_{S_\infty} \left< \delta \II_\rho - \frac{H_\rho}{2}\delta \I_\rho, \I_\rho \right>_{\I_\rho} da_{\rho} =
- \int_{S_\infty} \delta H^* + \left< \delta \I^* , \II_0^* \right>_{\I^*} da^*~.$$
It is also shown there (Remark 5.4) that $H^* = -K^*$. However, we have chosen $\I^*$ such that $K^* = -1$, it follows that $\delta H^*=0$.
\end{proof}

We need this last ingredient, proven in \cite{KS06}:
\begin{propo}[Krasnov-Schlenker \cite{KS06}]
Let $Z \in \QFS$ be a quasi-Fuchsian structure and $M$ be the associated quasi-Fuchsian $3$-manifold. 
Let $(S_\rho)_{\rho \geqslant 0}$ be the unique equidistant foliation associated to the Poincar\'e metric 
on $\partial_\infty^+ M$ and let $\II_0^*$
($= \II^* - \I^*$) be the trace-free part of $\II^*$ as above. 
Then \begin{equation}\II_0^* = \mathrm{Re}(\tau^{\sigma_{\cal F}}(Z))~.\end{equation}
\end{propo}

Recall that $\tau^{\sigma_{\cal F}} : \CPS \to T^*\TS$ is the identification of holomorphic affine bundles (given by
the Schwarzian parametrization) when one chooses the Fuchsian section $\sigma_{\cal{F}} : \TS \to \CPS$ as the ``zero section''
of $\CPS$ (see paragraph \ref{affinebundle}). In other words, $\tau^{\sigma_{\cal F}}(Z)$
is the holomorphic quadratic differential on $\partial_\infty^+ M$ given by the Schwarzian of
the projective structure $Z$ on $\partial_\infty^+ M$ relative to the uniformized projective structure.

As a consequence of this proposition and the previous theorem, we obtain:
\begin{theorem}\label{brodu}
Let $W : \AFS \to \R$ be the renormalized volume function on the almost-Fuchsian space, 
$\Psi : \AFS \to T^*\TS$ be the map defined in paragraph \ref{tobi}. Then 
\begin{equation}
dW = -\frac{1}{4} \mathrm{Re}\left[\Psi^* \xi + {(\tau^{\sigma_{\cal F}})}^* \xi\right]
\end{equation}
where $\xi$ is the canonical one-form on the complex cotangent space $T^*\TS$.
\end{theorem}

Note that this is an equality of real $1$-forms on $\AFS$. Taking the exterior derivative again yields
$0 = -\frac{1}{4} \mathrm{Re}\left[\Psi^* \omega_{\mathrm{can}} + {(\tau^{\sigma_{\cal F}})}^* \omega_{\mathrm{can}}\right]$.
Recall that $\omega_{\cal H} = \mathrm{Re}(\Psi^*\omega_{\mathrm{can}})$ (see paragraph \ref{tobi}) and 
${(\tau^{\sigma_{\cal F}})}^* \omega_{\mathrm{can}} = -i(\omega_G - p^* \omega_{WP})$ (Corollary \ref{blob}). 
The identification of symplectic structures follows:
\begin{theorem}\label{main4}
On the almost-Fuchsian space $\AFS$, 
\begin{equation}
\omega_{\cal H} = -\mathrm{Im}(\omega_G)~.
\end{equation}
\end{theorem}

\bibliographystyle{alpha}
\bibliography{Article2}

\end{document}